\documentclass[11pt,a4paper]{article}
\usepackage[utf8]{inputenc}
\usepackage[T1]{fontenc}
\usepackage[english]{babel}
\usepackage{amsmath, amssymb, amsthm, mathrsfs}
\usepackage{geometry}

\usepackage[colorlinks=true, linkcolor=blue, citecolor=blue, urlcolor=blue]{hyperref}

\newtheorem{theorem}{Theorem}[section]
\newtheorem{lemma}[theorem]{Lemma}
\newtheorem{proposition}[theorem]{Proposition}

\theoremstyle{definition}
\newtheorem{definition}[theorem]{Definition}
\newtheorem{example}[theorem]{Example}
\newtheorem{remark}[theorem]{Remark}

\title{On the validity of Tonelli's Theorem}
\author{P.~Roselli and M.~Willem}
\date{\today}

\begin{document}

\maketitle

\begin{abstract}
This work analyzes the validity of Tonelli's theorem, overcoming the traditional assumption of $\sigma$-finiteness. 
We prove that the existence and equality of iterated integrals for indicator functions is a necessary and sufficient condition to define a product measure that allow to extend Tonelli's Theorem to more general measure spaces, including $s$-finite measures. Finally, we characterize the semi-finiteness of such a product measure.
\end{abstract}

\section{Introduction}
The existence and equality of iterated integrals provides a natural tool for the definition of the product measure of a set $A\subseteq X\times Y$ contained in the $\sigma$-algebra $\Sigma_X\otimes\Sigma_Y$ generated by the measurable rectangles~$\Sigma_X\times\Sigma_Y$:
\begin{equation}\label{eq:misura da integrali iterati}
(\mu\otimes\nu)_{(A)}
=
\int_Y\left(\int_X {\chi_A}_{(x,y)}\, d\mu_{(x)}\right)d\nu_{(y)}
=
\int_X\left(\int_Y {\chi_A}_{(x,y)}\, d\nu_{(y)}\right)d\mu_{(x)}\ ,
\end{equation}
where $\chi_A$ is the indicator function of $A$, and $(X,\Sigma_X,\mu)$ and $(Y,\Sigma_Y,\nu)$ are measure spaces.
This definition of product measure is used, for example, in the classical texts of Saks~\cite{saks1937}, Halmos~\cite{Halmos1950}, Rudin~\cite{Rudin1987}, and Brokate and Kersting~\cite{brokate2015}.
Unfortunately, these authors assume a priori that the measures $\mu$ and $\nu$ are $\sigma$-finite.
But the existence and equality of the iterated integrals%~(\ref{eq:misura da integrali iterati})
\begin{equation}\label{eq:integrali iterati}
\int_Y\left(\int_X {\chi_A}_{(x,y)}\, d\mu_{(x)}\right)d\nu_{(y)}
=
\int_X\left(\int_Y {\chi_A}_{(x,y)}\, d\nu_{(y)}\right)d\mu_{(x)}\ ,
\end{equation}
holds also for pairs of $s$-finite measures, as observed by Falkner in~\cite{falkner2009review} and~\cite{falkner2019}.
Indeed, $\sigma$-finiteness is not necessary for the validity of~(\ref{eq:integrali iterati}).
In this work we define the notion of an \emph{abelian pair} of measures: a pair $(\mu,\nu)$ of measures is, by definition, \emph{abelian} if~(\ref{eq:integrali iterati}) holds for any set $A\in\Sigma_X\otimes\Sigma_Y$.
Theorem~\ref{thm:abelianità e funzioni} states that a pair $(\mu,\nu)$ is abelian if and only if for every function $f:X\times Y \rightarrow[0,\infty]$ that is $\Sigma_X\otimes\Sigma_Y$-measurable, the iterated integrals exist and coincide:
\begin{equation}\label{eq:identità tra integrali di funzioni}
\int_Y\left(\int_X f_{(x,y)}\, d\mu_{(x)}\right)d\nu_{(y)}
=
\int_X\left(\int_Y f_{(x,y)}\, d\nu_{(y)}\right)d\mu_{(x)}\ .
\end{equation}
%However, even in the case of finite measures ($\mu(X)<\infty$, $\nu(Y)<\infty$), the proof of~(\ref{eq:identità tra integrali di funzioni}) is not trivial.
%Instead, Theorem~\ref{thm:abelianità 1} proposes a characterization of abelian pairs that makes the existence and equality of the two preceding integrals immediate in the case of finite measures.
In Section~\ref{sec:misura prodotto} we define the product measure $\mu\otimes\nu$ associated with the abelian pair $(\mu,\nu)$ through the relation~(\ref{eq:misura da integrali iterati}).
In this optimal framework, the classical theorems of Tonelli and Fubini then follow from~(\ref{eq:identità tra integrali di funzioni}).
In Section~\ref{sect:semifinitezza} we prove that for an abelian pair $(\mu,\nu)$, the measures $\mu$ and $\nu$ are semi-finite if and only if the measure $\mu\otimes\nu$ is semi-finite.

\section{Iterated integrals}\label{sect:classe commutativa}

Let $\Omega$ be a set, we denote by the symbol $\Sigma_\Omega$ a $\sigma$-algebra of subsets of $\Omega$. We will call the pair $(\Omega, \Sigma_\Omega)$ a \emph{measurable space}. 
Let $\overline{\mathbb{R}} = \mathbb{R} \cup \{-\infty, +\infty\}$. Given a measurable space $(\Omega, \Sigma_\Omega)$, a function $u: \Omega \to \overline{\mathbb{R}}$ is said to be \emph{$\Sigma_\Omega$-measurable} if for every $t\in\mathbb{R}$ one has $\{\omega \in \Omega \mid u(\omega) > t\} \in \Sigma_\Omega$.
Given two measurable spaces $(X, \Sigma_X)$ and $(Y, \Sigma_Y)$, we denote by $\Sigma_X \otimes \Sigma_Y$ the $\sigma$-algebra generated by the measurable rectangles $\Sigma_X \times \Sigma_Y$, that is, the smallest $\sigma$-algebra containing the family $\Sigma_X \times \Sigma_Y$ of rectangles $E \times F$ (where $E\in\Sigma_X$ and $F\in\Sigma_Y$).

\begin{remark}
It will be useful to characterize $\Sigma_X \otimes \Sigma_Y$ as the smallest family $\mathscr{D}$ of subsets of $X \times Y$ that satisfies 
\begin{enumerate}
    \item[($\alpha$)] $\Sigma_X \times \Sigma_Y \subseteq \mathscr{D}$;
    \item[($\beta$)] $\mathscr{D}$ is closed under countable unions: if $\{A_n\}_{n=1}^\infty\subset\mathscr{D}$, then $\displaystyle \bigcup_{n=1}^\infty A_n\in\mathscr{D}$;
    \item[($\gamma$)] $\mathscr{D}$ is closed under set complementation: if $A\in \mathscr{D}$, then $A^c=(X\times Y)\setminus A\in\mathscr{D}$.
\end{enumerate}
\end{remark}

Let $A \subseteq X \times Y$ be a subset of the product; we define the sections of $A$ as follows:
\begin{itemize}
	\item $A^y = \{x \in X \mid (x,y) \in A\}$;
	\item $A_x = \{y \in Y \mid (x,y) \in A\}$.
\end{itemize}

\begin{lemma}\label{lem:sezioni misurabili}
If $A\in \Sigma_X \otimes \Sigma_Y$ then, 
\begin{itemize}
	\item for every $y\in Y$, $A^y\in\Sigma_X$;
	\item for every $x\in X$, $A_x\in\Sigma_Y$.
\end{itemize}
\end{lemma}
\begin{proof}
See Theorem A of Section~34 in~\cite{Halmos1950}.
\end{proof}

\begin{definition}
A \emph{measure} on the measurable space $(\Omega,\Sigma_\Omega)$ is a function $\mu: \Sigma_\Omega\longrightarrow [0,+\infty]$ such that
\begin{itemize}
	\item $\mu_{(\emptyset)}=0$;
	\item for every sequence $(A_n)$ in $\Sigma_\Omega$ of disjoint sets, 
	$\displaystyle \mu_{\big(\bigsqcup_{n} A_n\big)}=\sum_{n}\mu_{(A_n)}$.
\end{itemize}
In such a case, the triple $(\Omega,\Sigma_\Omega,\mu)$ is called a \emph{measure space}.
\end{definition}

\begin{definition}\label{def:classe_A}
Let $(X, \Sigma_X, \mu)$ and $(Y, \Sigma_Y, \nu)$ be measure spaces. We denote by $\mathscr{A}(\mu, \nu)$ the class of subsets $A \in \Sigma_X \otimes \Sigma_Y$ such that
\begin{enumerate}
	\item[(i)] the function $y \mapsto \mu_{(A^y)}$ is $\Sigma_Y$-measurable;
	\item[(ii)] the function $x \mapsto \nu_{(A_x)}$ is $\Sigma_X$-measurable;
	\item[(iii)] $\displaystyle 
	\int_Y \mu_{(A^y)}\, d\nu_{(y)} = \int_X \nu_{(A_x)}\, d\mu_{(x)}$.
\end{enumerate}
\label{def:commutano}
The pair $(\mu,\nu)$ is said to be \emph{abelian} if
\[
\mathscr{A}(\mu,\nu)
=
\Sigma_X\otimes\Sigma_Y
\]
\end{definition}

\begin{example}\label{exe:non abeliana 1}
Let $X = Y = [0,1]$ and let $\Sigma_X =\Sigma_Y= \mathcal{B}\big([0,1]\big)$ be the $\sigma$-algebra of Borel sets of $[0,1]$.
Consider any set $S\subset[0,1]$ such that $S\not\in\mathcal{B}\big([0,1]\big)$, and define for every $B\in\mathcal{B}\big([0,1]\big)$
\begin{center}
$
\mu_{(B)}
=
\nu_{(B)}
=
\#_{(B\cap S)}$,
\end{center}
where $\#$ denotes the counting measure. Consider the diagonal of the unit square: $\Delta = \big\{(x,x)\mid x\in[0,1]\big\}$. Since $\Delta$ is closed, $\Delta\in\mathcal{B}\big([0,1]\big)\otimes\mathcal{B}\big([0,1]\big)$. Therefore, for every $y\in[0,1]$, we have that
\begin{center}
$
\mu_{(\Delta^y)}
=
{\chi_S}_{(y)}
$.
\end{center}
Since $S\not\in\mathcal{B}\big([0,1]\big)$, then the function
\begin{center}
$
y \longmapsto \mu_{(\Delta^y)}={\chi_S}_{(y)}
$
\end{center}
is not $\mathcal{B}\big([0,1]\big)$-measurable and condition (i) of Definition~\ref{def:commutano} is not satisfied. Hence, the pair $(\mu,\mu)$ is not abelian. 
\end{example}

\begin{example}\label{exe:non abeliana 3}
Let $X = Y = [-1,1]$ and let $\Sigma_X =\Sigma_Y= \mathcal{B}\big([-1,1]\big)$.
For every $B\in\mathcal{B}\big([0,1]\big)$ we define
\begin{center}
$
\mu_{(B)}
=
\nu_{(B)}
=
\lambda_{(B\cap [-1,0])}+\#_{(B\cap (0,1])}
$,
\end{center}
where $\lambda$ denotes the Lebesgue measure. Consider the Borel set
\begin{center}
$
A
=
\{ (x,y)\in [-1,0]\times[0,1]\mid y=x+1 \}
$.
\end{center}
For every $y\in[0,1]$ and for every $x\in[-1,0]$, we have that 
\begin{center}
$
\mu_{(A^y)}=0
$,
$
\mu_{(A_x)}=1
$.
\end{center}
Therefore, the iterated integrals are well-defined. However, they are not equal
\begin{center}
$
\displaystyle
\int_{Y} \mu_{(A^y)}\ d\mu_{(y)}
=0 \ne 1 =
\int_{X} \mu_{(A_x)}\ d\mu_{(x)}
$.
\end{center}
Therefore, condition (iii) of Definition~\ref{def:commutano} is not satisfied, and the pair $(\mu,\mu)$ is not abelian. 
\end{example}

\begin{remark}
The previous two examples show that the abelianity relation between measures, despite being symmetric, is not reflexive. 
\end{remark}

\begin{example}
Consider the measure space $(X, \Sigma_X, \mu)=(\mathbb{N}, \mathcal{P}(\mathbb{N}), \#)$, and $(Y, \Sigma_Y, \nu)$ \emph{any} measure space. Let $A \in \mathcal{P}(\mathbb{N}) \otimes \Sigma_Y$. By Lemma~\ref{lem:sezioni misurabili}, for every $n \in \mathbb{N}$ we have that 
\begin{center}
$
A_n = \big\{y \in Y \mid (n, y) \in A\big\}\in\Sigma_Y
$.
\end{center}
The Monotone Convergence Theorem implies that 
\begin{align*}
\int_{\mathbb{N}} \nu_{(A_n)} \, d\#_{(n)} 
& = \sum_{n=1}^\infty \nu_{(A_n)}\\
& = \sum_{n=1}^\infty \int_{Y} {\chi_{A_n}}_{(y)} \, d\nu_{(y)}\\
& =  \int_{Y} \sum_{n=1}^\infty{\chi_{A_n}}_{(y)} \, d\nu_{(y)}\\
& = \int_{Y} \#_{(A^y)} \, d\nu_{(y)}
\end{align*}
Consequently, the pair $(\#,\nu)$ is abelian.
\end{example}

\begin{lemma}\label{lem:proprietà A}
Let $(X, \Sigma_X, \mu)$ and $(Y, \Sigma_Y, \nu)$ be measure spaces. Then,
\begin{enumerate}
	\item[($\alpha$)] $\Sigma_X\times\Sigma_Y\subset \mathscr{A}(\mu,\nu)$
	\item[($\beta$)] for every sequence $(A_n)$ of disjoint elements in $\mathscr{A}(\mu,\nu)$ (that is, $i\ne j \Rightarrow A_i\cap A_j =\emptyset$), we have that $\displaystyle A=\bigsqcup_{n=1}^\infty A_n\in\mathscr{A}(\mu,\nu)$, and 
	\[
  \int_Y \mu_{(A^y)} \, d\nu_{(y)} 
	= \sum_{n=1}^\infty \int_Y \mu_{\big((A_n)^y\big)} d\nu_{(y)} 
  = \sum_{n=1}^\infty \int_X \nu_{\big((A_n)_x\big)} d\mu_{(x)} 
	= \int_X \nu_{(A_x)} \, d\mu_{(x)}
	\]
\end{enumerate}
\end{lemma}
\begin{proof}[Proof of ($\alpha$)] Let $R = E \times F \in \Sigma_X \times \Sigma_Y$. Fixing $y \in Y$, the section is $R^y = E$ if $y \in F$, and $\emptyset$ if $y \notin F$. Thus, $\mu_{(R^y)} = \mu_{(E)}\chi_F(y)$. Since $F \in \Sigma_Y$, the function $y \mapsto \mu_{(R^y)}$ is $\Sigma_Y$-measurable.
Similarly, fixing $x \in X$, the section is $R_x = F$ if $x \in E$, and $\emptyset$ otherwise.
Thus, $\nu_{(R_x)} = \nu_{(F)}\chi_E(x)$, which is $\Sigma_X$-measurable.
Computing the integrals:
\begin{align*}
    \int_Y \mu_{(R^y)}\, d\nu_{(y)} &= \int_Y \mu_{(E)}{\chi_F}_{(y)}\, d\nu_{(y)} = \mu_{(E)}\nu_{(F)} \\
    \int_X \nu_{(R_x)}\, d\mu_{(x)} &= \int_X \nu_{(F)}{\chi_E}_{(x)}\, d\mu_{(x)} = \nu_{(F)}\mu_{(E)}
\end{align*}
The integrals are equal, therefore $R \in \mathscr{A}(\mu,\nu)$.
\end{proof}

\begin{proof}[Proof of ($\beta$)] Let $A = \bigsqcup_{n=1}^\infty A_n$ be a countable disjoint union of elements in $\mathscr{A}(\mu,\nu)$. Fixing $y \in Y$, the section $A^y = \bigsqcup_{n=1}^\infty (A_n)^y$ is a disjoint union of elements of $\Sigma_X$.
By the $\sigma$-additivity of $\mu$:
$$ 
\mu_{(A^y)} = \sum_{n=1}^\infty \mu_{\big((A_n)^y\big)} $$
Being the sum of a series of measurable functions, the function $y \mapsto \mu_{(A^y)}$ is $\Sigma_Y$-measurable.
Similarly, $x \mapsto \nu_{(A_x)} = \sum_{n=1}^\infty \nu_{\big((A_n)_x\big)}$ is $\Sigma_X$-measurable.
Applying the Monotone Convergence Theorem to swap the series and the integral, we obtain:
\begin{align*}
    \int_Y \mu_{(A^y)} \, d\nu_{(y)} 
		&= \int_Y \sum_{n=1}^\infty \mu_{\big((A_n)^y\big)} d\nu_{(y)} = \sum_{n=1}^\infty \int_Y \mu_{\big((A_n)^y\big)} d\nu_{(y)} \\
    &= \sum_{n=1}^\infty \int_X \nu_{\big((A_n)_x\big)} d\mu_{(x)} = \int_X \sum_{n=1}^\infty \nu_{\big((A_n)_x\big)} d\mu_{(x)} = \int_X \nu_{(A_x)} \, d\mu_{(x)}
\end{align*}
Thus $A \in \mathscr{A}(\mu,\nu)$.
\end{proof}

\begin{theorem}\label{thm:abelianità 1}
Let $(X, \Sigma_X, \mu)$ and $(Y, \Sigma_Y, \nu)$ be two measure spaces. The following statements are equivalent:
\begin{enumerate}
	\item[(a)] for every $A\in\mathscr{A}(\mu,\nu)$, $A^c\in\mathscr{A}(\mu,\nu)$;
	\item[(b)] the pair $(\mu,\nu)$ is abelian.
\end{enumerate}
\end{theorem}

\begin{proof}[Proof of (a)$\Rightarrow$(b)]
The family of measurable rectangles $\mathscr{R} = \Sigma_X \times \Sigma_Y$ is stable under finite intersections and generates $\Sigma_X \otimes \Sigma_Y$.
By hypothesis (a), the class $\mathscr{A}(\mu,\nu)$ is closed under complementation. 
Combining this hypothesis with the results of the previous Lemma, we know that $\mathscr{A}(\mu,\nu)$ contains the entire space $X \times Y$ \big(thanks to property ($\alpha$) of Lemma~\ref{lem:proprietà A}\big), and is closed under countable disjoint unions.
These three properties define a \textit{Dynkin System}.
Since $\mathscr{A}(\mu,\nu)$ is a Dynkin system that includes $\mathscr{R}$, by Proposition 7.2 in~\cite{brokate2015}, we deduce that
%the inclusion extends to the entire generated $\sigma$-algebra:
$$ \mathscr{A}(\mu,\nu) = \Sigma_X \otimes \Sigma_Y \ .
$$
\end{proof}

\begin{proof}[Proof of (b)$\Rightarrow$(a)]
By definition, $\mathscr{A}(\mu,\nu)=\Sigma_X\otimes\Sigma_Y$, and $\Sigma_X\otimes\Sigma_Y$ is a $\sigma$-algebra.
\end{proof}

\begin{theorem} \label{thm:finite_commute} \label{lem:complementari}
If $\mu$ and $\nu$ are both finite measures, the pair $(\mu,\nu)$ is abelian.
\end{theorem}
\begin{proof}
The section of the complement is $(A^c)_x = Y \setminus A_x$.
Since $\nu$ is finite,
$$ 
\nu_{\big((A^c)_x\big)}
=
\nu_{(Y)} - \nu_{(A_x)} \ .
$$ 
Integrating, we obtain
$$
\int_X \nu_{\big((A^c)_x\big)} \, d\mu_{(x)} = \mu_{(X)}\nu_{(Y)} - \int_X \nu_{(A_x)} \, d\mu_{(x)} \ .
$$
Proceeding symmetrically and recalling that $A \in \mathscr{A}(\mu, \nu)$, the two expressions coincide.
Thus, $A^c \in \mathscr{A}(\mu, \nu)$.
\end{proof}

A generalization of finite measures is the class of \emph{s-finite} measures. A measure is said to be s-finite if it can be expressed as a \emph{countable} sum of finite measures all defined on the same measurable space $(\Omega,\Sigma_\Omega)$.

\begin{theorem}\label{cor:s-finite_commute}
If $\mu$ and $\nu$ are two s-finite measures, then the pair $(\mu,\nu)$ is abelian.
\end{theorem}
\begin{proof}
By hypothesis, $\mu = \sum_{i=1}^\infty \mu_i$ and $\nu = \sum_{j=1}^\infty \nu_j$, where $\mu_i$ and $\nu_j$ are finite measures on $(X,\Sigma_X)$ and on $(Y,\Sigma_Y)$, respectively.
Let $A \in \Sigma_X \otimes \Sigma_Y$. We expand the first iterated integral making use of monotone convergence:
$$
\int_X \nu_{(A_x)} \, d\mu_{(x)} 
=
\int_X \left( \sum_{j=1}^\infty {\nu_j}_{(A_x)} \right) d \left(\sum_{i=1}^\infty \mu_i\right)\kern-5pt\,_{(x)}
=
\sum_{i=1}^\infty \sum_{j=1}^\infty \int_X {\nu_j}_{(A_x)} \, d{\mu_i}_{(x)} \ .
$$
Since $\mu_i$ and $\nu_j$ are finite measures, by Theorem \ref{thm:finite_commute} every pair $(\mu_i,\nu_j)$ is abelian; hence
$$ 
\int_X \nu_{(A_x)} \, d\mu_{(x)} 
=
\sum_{i=1}^\infty \sum_{j=1}^\infty \int_Y {\mu_i}_{(A^y)} \, d{\nu_j}_{(y)} \ .
$$
Passing the series under the integral sign (again by monotone convergence), we obtain the expression $\int_Y \mu_{(A^y)} \, d\nu_{(y)}$.
Both functions are countable sums of measurable functions, thus they themselves are measurable. It follows that $A \in \mathscr{A}(\mu, \nu)$.
\end{proof}

For $(\Sigma_X \otimes \Sigma_Y)$-measurable functions $f:X\times Y\to[0,+\infty]$, something similar happens to what we have already recalled for the elements $A\in\Sigma_X \otimes \Sigma_Y$ in Lemma~\ref{lem:sezioni misurabili}.

\noindent\textbf{Notations.} Let $f:X\times Y\to Z$. For every $x\in X$ and for every $y\in Y$, we define
\[
f^y(x)=f(x,y)=f_x(y)\ .
\]

\begin{proposition}
Let $(X,\Sigma_X)$ and $(Y,\Sigma_Y)$ be measure spaces, and let $f:X\times Y\to[0,+\infty]$ be a $(\Sigma_X\otimes\Sigma_Y)$-measurable function. Then
\begin{itemize}
	\item for every $y\in Y$ the function $f^y$ is $\Sigma_X$-measurable on $X$;
	\item for every $x\in X$ the function $f_x$ is $\Sigma_Y$-measurable on $Y$.
\end{itemize}
\end{proposition}
\begin{proof}
See Theorem B of Section~34 in~\cite{Halmos1950}.
\end{proof}

\noindent
However, without particular assumptions on the measures $\mu$ and $\nu$, it is not guaranteed that the functions $y \mapsto \int_X f^y(x)\, d\mu_{(x)}$ and $x \mapsto \int_Y f_x(y)\, d\nu_{(y)}$ are still measurable.

\begin{definition}\label{def:classe_F}
Let $(X, \Sigma_X, \mu)$ and $(Y, \Sigma_Y, \nu)$ be measure spaces. We denote by $\mathcal{F}(\mu,\nu)$ the set of $(\Sigma_X \otimes \Sigma_Y)$-measurable functions $f:X\times Y\to[0,+\infty]$ such that 
\begin{enumerate}
	\item[(i)] the function $y \mapsto \int_X f^y(x)\, d\mu_{(x)}$ is $\Sigma_Y$-measurable on $Y$;
	\item[(ii)] the function $x \mapsto \int_Y f_x(y)\, d\nu_{(y)}$ is $\Sigma_X$-measurable on $X$;
	\item[(iii)] 
		$\displaystyle
\int_Y\left(\int_X f^y(x)\, d\mu_{(x)}\right)d\nu_{(y)}
	=
\int_X\left(\int_Y f_x(y)\, d\nu_{(y)}\right)d\mu_{(x)}
		$
\end{enumerate}
\end{definition}

\begin{lemma}\label{lem:monotonia}
Let $(\Omega,\Sigma_\Omega)$ be a measure space and let $\mathcal{K}$ be a set of functions $f:\Omega\to[0,+\infty]$ such that 
\begin{itemize}
	\item[(a)] for every $A\in\Sigma_\Omega$, $\chi_A \in\mathcal{K}$;
	\item[(b)] for every $f,g\in\mathcal{K}$, and $s,t\in[0,+\infty]$
	\[sf+tg\in\mathcal{K}\ ;\]
	\item[(c)] for every sequence $(f_n)$ of elements of $\mathcal{K}$, such that $f_1\le f_2\le \cdots$, we have that
	\[\sup_n f_n\in\mathcal{K}\ ;\]
\end{itemize}
then
\[
\big\{
f:\Omega\to[0,+\infty] \mid f \textrm{ is $\Sigma_\Omega$-measurable }
\big\}
\subseteq \mathcal{K}\ .
\]
\end{lemma}
\begin{proof}
Let $f:\Omega\to[0,+\infty]$ be a $\Sigma_\Omega$-measurable function. For every integer $n\ge 1$, we define the function:
\[
f_n = \frac{1}{2^n} \sum_{k=1}^{\infty} \chi_{\left\{f>\frac{k}{2^n} \right\}}\ .
\]
By properties $(a)$, $(b)$ and $(c)$, $f_n \in \mathcal{K}$ for every $n$.
Thanks to property~$(c)$, we deduce that $f = \sup_n f_n \in\mathcal{K}$ (see, for example,~\cite{willem2013}).
\end{proof}

\begin{theorem}\label{thm:abelianità e funzioni} 
Let $(X, \Sigma_X, \mu)$ and $(Y, \Sigma_Y, \nu)$ be two measure spaces. Then, the following statements are equivalent:
\begin{itemize}
	\item[(a)] The pair $(\mu,\nu)$ is abelian;
	\item[(b)]
	$\mathcal{F}(\mu,\nu)=\big\{f:X\times Y\to[0,+\infty] \mid f \textit{ is } (\Sigma_X\otimes\Sigma_Y)\textit{-measurable}\big\}$.
\end{itemize}
\end{theorem}
\begin{proof}[Proof of (b) $\Rightarrow$ (a)]
Let $A \in \Sigma_X \otimes \Sigma_Y$. If hypothesis~(b) is satisfied, we deduce that~$\chi_A \in \mathcal{F}(\mu,\nu)$.
Comparing Definition~\ref{def:classe_A} with Definition~\ref{def:classe_F}, we have that
\[
A\in \mathscr{A}(\mu,\nu)
\ \Leftrightarrow
\chi_A \in \mathcal{F}(\mu,\nu)\ .
\]
This implies that~$\mathscr{A}(\mu,\nu) = \Sigma_X \otimes \Sigma_Y$. That is, the pair $(\mu,\nu)$ is abelian.
\end{proof}

\begin{proof}[Proof of (a) $\Rightarrow$ (b)]
Let us consider Lemma~\ref{lem:monotonia} with $\Omega=X\times Y$, $\Sigma_\Omega=\Sigma_X\otimes\Sigma_Y$ and~$\mathcal{K}=\mathcal{F}(\mu,\nu)$.
By hypothesis $(a)$, we have that for every $A \in \Sigma_X \otimes \Sigma_Y$, $\chi_A \in \mathcal{F}(\mu,\nu)$.
By the linearity of integrals, hypothesis $(b)$ of Lemma~\ref{lem:monotonia} is satisfied.
Let us now consider a sequence $(f_n)$ of elements of $\mathcal{F}(\mu,\nu)$ such that $f_1\le f_2\le \cdots$ and define $\displaystyle f=\sup_n f_n=\lim_n f_n$.
The Monotone Convergence Theorem implies that, for every $y\in Y$, 
\[
\int_X
f^y(x)\ d\mu_{(x)}
=
\int_X
\left[\sup_n f_n(x)\right]^y\ d\mu_{(x)}
=
\sup_n\int_X
\left[f_n(x)\right]^y\ d\mu_{(x)}
\]
From this it follows that the function $\displaystyle y\mapsto \int_X f^y(x)\ d\mu_{(x)}$ is $\Sigma_Y$-measurable on $Y$.
Similarly, the function $\displaystyle x\mapsto \int_X f_x(y)\ d\nu_{(y)}$ is $\Sigma_X$-measurable on $X$.
Finally, using the Monotone Convergence Theorem again, we have that
\begin{align*}
\int_Y\left(\int_X f^y\, d\mu_{(x)}\right)d\nu_{(y)} 
&= \int_Y\left[\sup_n\left(\int_X f_n^y(x)\, d\mu_{(x)}\right)\right]d\nu_{(y)} \\
&= \sup_n\int_Y\left(\int_X f_n^y(x)\, d\mu_{(x)}\right)d\nu_{(y)} \\
&= \sup_{n} \int_X\left(\int_Y (f_n)_x(y)\, d\nu_{(y)}\right)d\mu_{(x)} \\
&= \int_X\left(\int_Y f_x(y)\, d\nu_{(y)}\right)d\mu_{(x)} \ .
\end{align*}
This proves that $\displaystyle f=\sup_n f_n \in \mathcal{F}(\mu,\nu)$. Thanks to Lemma~\ref{lem:monotonia}, we can conclude that\\ 
$\mathcal{F}(\mu,\nu)=\big\{f:X\times Y\to[0,+\infty] \mid f \textit{ is } (\Sigma_X\otimes\Sigma_Y)\textit{-measurable}\big\}$.
\end{proof}

\section{The product measure}\label{sec:misura prodotto}

From Theorem~\ref{thm:abelianità e funzioni} it follows that, if Tonelli's Theorem holds, the pair of measures involved is abelian. To prove Tonelli's Theorem in this optimal framework, we define the product measure of an abelian pair of measures.

\begin{definition}\label{def:misura prodotto}
Let $(\mu,\nu)$ be an abelian pair of measures. The \emph{product measure} of $A \in \Sigma_X \otimes \Sigma_Y$ is defined by
\begin{equation}\label{eq:misura prodotto}
(\mu \otimes \nu)_{(A)} 
=
\int_X \nu_{(A_x)} \, d\mu_{(x)}
=
\int_Y \mu_{(A^y)} \, d\nu_{(y)}
\end{equation}
\end{definition}

\begin{proposition}\label{prop:buona_positura_prodotto}
Let $(X,\Sigma_X,\mu)$ and $(Y,\Sigma_Y,\nu)$ be measure spaces such that the pair $(\mu,\nu)$ is abelian. Then the function $\mu \otimes \nu: \Sigma_X \otimes \Sigma_Y \to [0, +\infty]$ defined by~(\ref{eq:misura prodotto}) is a measure on the measurable space $(X \times Y, \Sigma_X \otimes \Sigma_Y)$.
\end{proposition}

\begin{proof}
It is sufficient to prove that $\mu\otimes\nu$ is $\sigma$-additive.
Let $\{A_n\}_{n=1}^\infty \subset \Sigma_X \otimes \Sigma_Y$ be a sequence of disjoint measurable sets. From Lemma~\ref{lem:proprietà A} it follows that the set $\displaystyle A=\bigsqcup_{n=1}^\infty A_n$ satisfies
\begin{align*}
(\mu \otimes \nu)_{(A)} 
& =
\int_Y \mu_{(A^y)}\ d\nu_{(y)}\\
& =
\sum_{n=1}^\infty
\int_Y \mu_{\big((A_n)^y\big)}\ d\nu_{(y)}\\
& =
\sum_{n=1}^\infty
(\mu \otimes \nu)_{(A_n)} \ .
\end{align*}

\end{proof}

\noindent
We recall that the outer measure of a set $A\in X\times Y$ is defined as follows
\[
(\mu\times\nu)^*_{(A)}
=
\inf
\left\{
\sum_{i=1}^\infty
\mu_{(B_n)}\nu_{(C_n)} 
\mid
(B_n)\subset\Sigma_X \ ,  \ (C_n)\subset\Sigma_Y \ ,\
A\subseteq\bigcup_{n=1}^\infty B_n\times C_n
\right\}\ ,
\]
with the convention that $0\cdot \infty=\infty\cdot 0=0$.

\begin{proposition}
Assume tha $(\mu,\nu)$ is an abelian pair of measures.
Let $A\in\Sigma_X \otimes\Sigma_Y$. Then
\[
(\mu\otimes\nu)_{(A)}
\le
(\mu\times\nu)^*_{(A)}\ .
\]
\end{proposition}
\begin{proof}
Let $(B_n)\subset\Sigma_X$ and $(C_n)\subset\Sigma_Y$ be such that 
$A\subseteq\bigcup_{n=1}^\infty B_n\times C_n$.
Since
\[
\chi_A
\le
\sum_{n=1}^\infty
\chi_{(B_n\times C_n)}\ ,
\]
the Monotone Convergence Theorem implies that
\begin{align*}
(\mu\otimes\nu)_{(A)}
& =
\int_{X\times Y}\chi_A\ d(\mu\otimes\nu)\\
& \le 
\int_{X\times Y}\sum_{n=1}^\infty
\chi_{(B_n\times C_n)}\ d(\mu\otimes\nu)\\
& =
\sum_{n=1}^\infty\int_{X\times Y}
\mu_{(B_n)}\nu_{(C_n)}\ .
\end{align*}
\end{proof}

\begin{example}\label{exe:di Falkner}
An example of the strict inequality is found in~\cite{falkner2019}, an example that N.~Falkner attributes to P.~J.~Fitzsimmons.
Consider the space $X=Y=[0,1]$ and we define on every element $A\in\mathcal{B}\big([0,1]\big)$, the $\sigma$-algebra of Borel sets on $[0,1]$, the s-finite measure
\begin{center}
$
\displaystyle
\mu_{(A)}
=
\nu_{(A)}
=
\sum_{n=1}^\infty
\lambda_{(A)}
$,
\end{center}
where $\lambda$ is the Lebesgue measure.
It is clear that $\mu_{(A)}=0$ if $\lambda_{(A)}=0$, and that $\mu_{(A)}=\infty$ if $\lambda_{(A)}>0$.
Let us calculate the product measure $\mu\otimes\nu$ of the diagonal $\Delta = \big\{(x,x) \mid x \in [0,1]\big\}$.
Since $\mu_{(\Delta^y)}
=0
=\nu_{(\Delta_x)}$,
we have that $(\mu\otimes\nu)_{(\Delta)}=0$.
Therefore,
\begin{center}
$\displaystyle
{(\mu\otimes\nu)}_{(\Delta)}=0<\infty={(\mu\times\nu)^*}_{\kern-4pt(\Delta)}
$,
\end{center}
%and the measure $\mu\otimes\nu$ cannot be defined using Carathéodory's construction.
and the measure $\mu\otimes\nu$ is not equal to the measure defined by Carathéodory’s extension
Theorem.

 %See also the approach to product measures in~\cite{royden1988} and~\cite{tao2011measure} independent from Carathéodory outer measure.
\end{example}

\begin{theorem}[Tonelli's Theorem]\label{thm:Tonelli}
Let $(X,\Sigma_X,\mu)$ and $(Y,\Sigma_Y,\nu)$ be measure spaces such that the pair~$(\mu,\nu)$ is abelian.
If the function $f: X \times Y \to [0, +\infty]$ is $(\Sigma_X \otimes \Sigma_Y)$-measurable, then
$$ 
\int_{X\times Y} f\ d(\mu\otimes\nu)
=
\int_X \left( \int_Y f(x,y) \, d\nu(y) \right) d\mu(x) 
=
\int_Y \left( \int_X f(x,y) \, d\mu(x) \right) d\nu(y) \ .
$$
\end{theorem}

\begin{proof}
The proof, using Lemma~\ref{lem:monotonia}, is similar to the proof of Theorem~\ref{thm:abelianità e funzioni}.
\end{proof}

Fubini's Theorem, which extends the analysis to summable functions of mixed sign, is obtained with the usual procedures of decomposing the function into a `positive part' and a `negative part'.

\begin{theorem}[Fubini's Theorem]
Let $(X,\Sigma_X,\mu)$ and $(Y,\Sigma_Y,\nu)$ be measure spaces such that the pair~$(\mu,\nu)$ is abelian.
If $u\in L^1(X \times Y, d(\mu\otimes\nu))$, then the sections of $u$ are almost everywhere integrable and the equality holds:
$$ 
\int_{X\times Y} u\ d(\mu\otimes\nu)
=
\int_X \left( \int_Y u(x,y) \, d\nu(y) \right) d\mu(x) 
=
\int_Y \left( \int_X u(x,y) \, d\mu(x) \right) d\nu(y) \ .
$$
\end{theorem}

\section{Semi-finiteness}\label{sect:semifinitezza}

\begin{definition}
The measure space $(\Omega,\Sigma_\Omega,\mu)$ (in short, $\mu$) is semi-finite if, for every $A\in\Sigma_\Omega$ such that $\mu_{(A)}>0$, there exists a $B\subseteq A$ such that $B\in\Sigma_\Omega$ and $0<\mu_{(B)}<\infty$.
\end{definition}

\begin{theorem}
Let $(\mu,\nu)$ be an abelian pair of measures such that $\mu_{(X)}>0$ and $\nu_{(Y)}>0$.
Then, the following statements are equivalent:
\begin{itemize}
	\item[(a)] $(\mu\otimes\nu)$ is semi-finite;
	\item[(b)] $\mu$ and $\nu$ are semi-finite.
\end{itemize}
\end{theorem}

\begin{proof}[Proof of $(b)\Rightarrow (a)$]
Let $A\in\Sigma_X\times\Sigma_Y$ such that
\[
(\mu\otimes\nu)_{(A)}
=
\int_X \nu_{(A_x)} \, d\mu_{(x)} >0\ .
\]
We define $\displaystyle D=\{x\in X \mid \nu_{(A_x)} >0\}$.
%Since $A \in \Sigma_X \otimes \Sigma_Y$ and the pair $(\mu,\nu)$ is abelian, we have $A \in \mathscr{A}(\mu,\nu)$.
%Consequently, by condition (ii) of Definition~\ref{def:classe_A}, the function $x \mapsto \nu_{(A_x)}$ is $\Sigma_X$-measurable.
%Thus, since $D$ is the preimage of the Borel interval $(0, +\infty]$ through a measurable function, it is in $\Sigma_X$.
Since $\mu_{(D)}>0$, the semi-finiteness of $\mu$ implies the existence of $E\subseteq D$ such that $E\in\Sigma_X$ and $0<\mu_{(E)}<\infty$.
The set $B=A\cap (E\times Y)$ satisfies
\begin{align*}
0
& < \int_E \nu_{(A_x)}\ d\mu_{(x)}\\
& = \int_X \nu_{(B_x)}\ d\mu_{(x)}\\
& = \int_Y \mu_{(B^y)}\ d\nu_{(y)}\ .
\end{align*}
We define $\displaystyle F=\{y\in Y \mid \mu_{(B^y)} >0\}$.
Since $\int_Y \mu_{(B^y)}\, d\nu_{(y)} > 0$, we have that~$\nu_{(F)}>0$; the semi-finiteness of $\nu$ implies the existence of $G\subseteq F$ such that $G\in\Sigma_Y$ and $0<\nu_{(G)}<\infty$.
The set $C=B\cap (X\times G)$ satisfies
\begin{align*}
0
& < \int_G \mu_{(B^y)}\ d\nu_{(y)}\\
& = \int_Y \mu_{(C^y)}\ d\nu_{(y)}\\
& = (\mu\otimes\nu)_{(C)}\ .
\end{align*}
Since 
$
(\mu\otimes\nu)_{(C)}
\le \mu_{(E)}\cdot\nu_{(G)}<\infty
$,
the proof is complete.
\end{proof}

\begin{proof}[Proof of $(a)\Rightarrow (b)$]
Let $C\in\Sigma_X$ such that $\mu_{(C)}>0$. We define $A=C\times Y$.
In this way, 
\begin{align*}
& A\in \Sigma_X\otimes\Sigma_Y \textrm{ and }\\
& (\mu\otimes\nu)_{(A)}=\mu_{(C)}\cdot \nu_{(Y)}>0\ .
\end{align*}
The semi-finiteness of $\mu\otimes\nu$ implies the existence of $B\subseteq A$  such that
\begin{center}
$
0 < (\mu\otimes\nu)_{(B)} = \int_Y \mu_{(B^y)}\ d\nu_{(y)}<\infty\ .
$
\end{center}
Therefore, there exists $y\in Y$ such that 
$0 < \mu_{(B^y)}<\infty$. Since $B^y\subseteq C$, the proof is complete.
\end{proof}

\section{Concluding remarks}

In this work, we have proved that the abelian property of the pair $(\mu,\nu)$ is a necessary and sufficient condition for the validity of Tonelli's Theorem.
Let us recall that, since 1937, in the classical book by Saks~\cite{saks1937}, it is assumed that the measures $\mu$ and $\nu$ are $\sigma$-finite. When the product measure is defined in a general setting by using Carathéodory's extension Theorem as in Rao~\cite{Rao2004} and Royden~\cite{RoydenFitzpatrick2023}, Tonelli's Theorem is also restricted to $\sigma$-finite measures (see~\cite{Rao2004} at p.\,385 and~\cite{RoydenFitzpatrick2023} at p.\,227).
Moreover, for a pair of s-finite measures,  the measure defined by Carathéodory’s extension
Theorem is not equal, in general, to the measure defined by iterated integrals.
Finally, the use of abelian pairs makes the theory technically simple, since the definitions and the theorems have similar structures.

\end{document}